\documentclass[letterpaper, 11 pt, twoside]{amsart}
\usepackage{srollens-en}[2011/03/03]
\usepackage[paper = letterpaper, left = 3.4cm, right = 3.4cm, headsep = 6mm,
footskip = 10mm, top = 34mm, bottom = 34mm, footnotesep=5mm, headheight =
2cm]{geometry}

\usepackage{color}
\usepackage[T1]{fontenc}
\usepackage{mathrsfs}
\usepackage{mathpazo}

\usepackage[active]{srcltx}

\usepackage{textcomp}

\usepackage[hypertexnames=false,
    pdftex,
	pdfpagemode=UseNone,
	breaklinks=true,
	extension=pdf,
	colorlinks=true,
	linkcolor=blue,
	citecolor=blue,
	urlcolor=blue,
]{hyperref}

\newcounter{saveenum}

\DeclareMathOperator{\Sp}{\mathrm{Sp}}
\newcommand{\red}{{red}}
\newcommand{\sm}{{\mathrm{reg}}}

\renewcommand{\restr}[1]{{\raisebox{-0.0\height}{$\mid_{#1}$}}}
\renewcommand{\epsilon}{\varepsilon}
\renewcommand{\phi}{\varphi}

\newcommand{\Bs}{\mathrm{Bs}}
\newcommand{\supp}{\mathrm{supp}}

\title[Lagrangian fibrations on hyperk\"ahler manifolds]{Lagrangian fibrations on
hyperk\"ahler manifolds\\ --- \\On a question of Beauville}

\author{Daniel Greb}
\address{Daniel Greb\\Institut f\"ur Mathematik\\Abteilung f\"ur Reine Mathematik\\Albert--Ludiwgs--Universit\"at Freiburg\\Eckerstra{\ss}e 1\\79104 Freiburg im Breisgau\\Germany}
\email{daniel.greb@math.uni-freiburg.de}

\author{Christian Lehn}
\address{Christian Lehn\\Institut de Recherche Math\'ematique  
Avanc\'ee\\ Universit\'e de Strasbourg\\
7 rue Ren\'e Descartes\\67084 Strasbourg Cedex\\France}
\email{lehn@math.unistra.fr}

\author{S\"onke Rollenske}
\address{S\"onke Rollenske\\Fakult\"at f\"ur Mathematik\\Universt\"at Bielefeld\\Universit\"atsstr. 25\\33615 Bielefeld\\Germany}
\email{rollenske@math.uni-bielefeld.de}

\begin{document}
\begin{abstract}
Let $X$ be a compact hyperk\"ahler manifold containing a complex torus $L$ as a Lagrangian subvariety. Beauville posed the question whether $X$ admits a Lagrangian fibration with fibre $L$. We show that this is indeed the case if $X$ is not projective. If $X$ is projective we find an almost holomorphic Lagrangian fibration with fibre $L$ under additional assumptions on the pair $(X, L)$, which can be formulated in topological or deformation--theoretic terms. Moreover, we show that for any such almost holomorphic Lagrangian fibration there exists a smooth good minimal model, i.e., a hyperk\"ahler manifold birational to $X$ on which the fibration is holomorphic.
\end{abstract}
\subjclass[2010]{53C26, 14D06, 14E30, 32G10, 32G05.}
\keywords{hyperk\"ahler manifold, Lagrangian fibration.\\\emph{Mots cl\'es:} vari\'et\'e hyperk\"ahl\'erienne, fibration lagrangienne}
\maketitle

\begin{center}
\begin{minipage}{.9\textwidth}
\begin{center}
\uppercase{ \bf Fibrations lagrangiennes des vari\'et\'es hyperk\"ahl\'eriennes\\  --- \\sur une question de Beauville}
\end{center}
\end{minipage}
\end{center}

\begin{abstract}
Soit $X$ une vari\'et\'e hyperk\"ahl\'erienne compacte contenant
un tore complexe $L$ comme sous--vari\'et\'e lagrangienne. A.
Beauville a pos\'e la question suivante : la vari\'et\'e $X$
admet-elle une fibration lagrangienne de fibre $L$? Nous d\'emontrons
que c{\textquotesingle}est le cas si $X$ n{\textquotesingle}est pas
projective. Si $X$ est projective nous montrons l'existence d'une
fibration lagrangienne presque holomorphe de fibre $L$ sous des
hypoth\`eses plus restrictives sur la paire $(X,L)$. Ces hypoth\`eses
peuvent se formuler de deux mani\`eres : en termes topologiques ou
gr\^ace \`a la th\'eorie des d\'eformations de $(X,L)$. Par ailleurs,
nous d\'emontrons que pour une telle fibration lagrangienne presque
holomorphe il y a toujours un bon mod\`ele minimal lisse,
c{\textquotesingle}est-\`a-dire une vari\'et\'e hyperk\"ahl\'erienne
birationelle \`a $X$ sur laquelle la fibration est holomorphe.
\end{abstract}

\tableofcontents

\section*{Introduction}
By the classical decomposition theorem of Beauville--Bogomolov, every compact K\"ahler manifold with vanishing first Chern class admits a finite cover which decomposes as a product of tori, Calabi--Yau manifolds, and hyperk\"ahler manifolds, see e.g.~\cite[Thm.~1]{beauville83}. While tori are quite well-understood, a  classification of Calabi--Yau and hyperk\"ahler manifolds is still far out of reach. Only in dimension 2, where Calabi--Yau and hyperk\"ahler manifolds coincide,  the theory of K3--surfaces provides a fairly complete picture.

Let now $X$ be a hyperk\"ahler manifold, that is, a compact, simply--connected K\"ahler manifold  $X$  such that $H^0(X, \Omega^2_X)$ is spanned by a holomorphic symplectic form  $\sigma$.
>From a differential geometric point of view hyperk\"ahler manifolds are Riemannian manifolds with holonomy the full unitary--symplectic group $\Sp(n)$.

An important step in the structural understanding of a manifold is to decide whether there is a fibration  $f\colon X\to B$  over a complex space of smaller dimension. For hyperk\"ahler manifolds it is known that in case such $f$ exists, it is a \emph{Lagrangian fibration}:  $\dim X = 2 \dim B$, and the holomorphic symplectic form $\sigma$ restricts to zero on the general fibre. Additionally, by the Arnold--Liouville theorem the general fibre is a smooth Lagrangian torus, see Section~\ref{sect: lagrfibr} for a detailed discussion.

In accordance with the case of K3--surfaces (and also motivated by mirror symmetry) a simple version of the so--called Hyperk\"ahler SYZ--conjecture\footnote{We refer the reader to \cite{verbitsky10} for a historical discussion concerning the emergence of this conjecture.}
asks if every hyperk\"ahler manifold can be deformed to a hyperk\"ahler manifold admitting a Lagrangian fibration. With this as a starting point, an approach to a rough classification of hyperk\"ahler manifolds  has been proposed, see e.g.~\cite{sawon03}. A more sophisticated version of the SYZ--conjecture is discussed in Section~\ref{sect:SYZ}.

Here we approach the question of existence of a Lagrangian fibration on a given hyperk\"ahler manifold $X$ under a geometric assumption proposed by Beauville \cite[Sect.~1.6]{beauville10}:

\begin{custom}[Question B]
Let $X$ be a hyperk\"ahler manifold and $L\subset X$ a
Lagrangian submanifold biholomorphic to a complex torus. Is $L$ a fibre of a (meromorphic) Lagrangian fibration $f\colon X\to B$?
\end{custom}

Building on work of Campana, Oguiso, and Peternell \cite{cop} we give a positive answer in case $X$ is not projective.
\begin{custom}[Theorem~\ref{mainnonproj}]
Let $X$ be a non--projective hyperk\"ahler manifold of dimension $2n$ containing a Lagrangian subtorus $L$. Then the algebraic dimension of $X$ is $n$, and there exists an algebraic reduction $f\colon  X \to B$ of $X$ that is a holomorphic Lagrangian fibration with fibre $L$.
\end{custom}

In the case of \emph{projective} hyperk\"ahler manifold $X$ containing a Lagrangian subtorus $L$, we work out a necessary and sufficient criterion for the existence of an almost holomorphic fibration with fibre $L$, i.e., for a slightly weaker positive answer to Beauville's question. 
\begin{custom}[Theorem~\ref{notstabproj}]
Let $X$ be a projective hyperk\"ahler manifold and $L\subset X$ a Lagrangian subtorus. Then the following are equivalent.
\begin{enumerate}
\item $X$ admits an almost holomorphic Lagrangian fibration with strong fibre $L$.
 \item The pair $(X,L)$ admits a small deformation $(X', L')$ with non-projective $X'$.
\item There exists an effective divisor $D$ on $X$ such that $c_1(\mathscr{O}_X(D)|_L) = 0 \in H^{1,1}\bigl(L, \, \mathbb{R} \bigr)$. 
\end{enumerate}
\end{custom}
Here, \emph{strong fibre} means that $f$ is holomorphic near $L$, and $L$ is a fibre of the corresponding holomorphic map. 
The proof of Theorem~\ref{notstabproj} consists of two major steps: First, assuming the existence of a small deformation of $(X,L)$ to a non-projective pair $(X',L')$, we use Theorem~\ref{mainnonproj} to produce a Lagrangian fibration with fibre $L'$ on $X'$ and then degenerate this fibration to an almost-holomorphic fibration on $(X,L)$ using relative Barlet spaces. Second, the existence of a small deformation to a non-projective pair $(X', L')$ is characterised in terms of periods in $H^2\bigl(X,\, \mathbb{C} \bigr)$. This finally leads to the condition on the existence of a special divisor, as stated in part \refenum{iii} of Theorem~\ref{notstabproj}.\footnote{After this article was written, Jun-Muk Hwang and Richard Weiss posted a preprint \cite{HwangWeiss} in which they prove that the criterion given in part \refenum{iii} of Theorem~\ref{notstabproj} is fulfilled for any projective hyperk\"ahler manifold $X$ containing a Lagrangian subtorus $L$. See also Remark \ref{rem: further devel}.} 


From the discussion above the question arises how far an almost holomorphic fibration is away from answering Beauville's question in the strong form. If $f\colon  X \dasharrow B$ is an almost holomorphic Lagrangian fibration, then it is natural to search for a holomorphic model of $f$ in the same birational equivalence class. This is done in the final section, where using the recent advances in higher--dimensional birational geometry (\cite{BCHM, HaconXu}) the following result is proven.
\begin{custom}[Theorem \normalfont{(see Theorem~\ref{thm: holomorphic model})}]
Let $X$ be a projective hyperk\"ahler manifold with an almost holomorphic Lagrangian fibration $f\colon X \dasharrow B$.  Then there exists a holomorphic model for $f$ on a birational hyperk\"ahler manifold $X'$. In other words, there is a commutative diagram
\[ \xymatrix{ X\ar@{-->}[d]_f\ar@{-->}[r] & X'\ar[d]^{f'}\\ B\ar@{-->}[r]& B'}\]
where $f'$ is a holomorphic Lagrangian fibration on $X'$ and the horizontal maps are birational.
\end{custom}
Theorem~\ref{thm: holomorphic model} proves a special version of the Hyperk\"ahler SYZ--conjecture. Related results were obtained by Amerik and Campana \cite[Thm.~3.6]{ame-cam08} in dimension four. Note furthermore that birational hyperk\"ahler manifolds are deformation--equivalent by work of Huybrechts \cite[Thm.~4.6]{huybrechts99}, so Theorem~\ref{thm: holomorphic model} might also lead to a new approach to the general case of the Hyperk\"ahler SYZ--conjecture.

The connection to this circle of ideas is also manifest in the following generalisation of a result of Matsushita, which we obtain as a corollary of Theorem~\ref{thm: holomorphic model}.
\begin{custom}[Theorem \ref{thm: almost holomorphic is Lagrangian}]
 Let $X$ be a projective hyperk\"ahler manifold and $f\colon X \dasharrow B$ an almost holomorphic map with connected fibres onto a normal projective variety $B$. If $0<\dim B< \dim X$, then $\dim B=\frac 1 2 \dim X$, and $f$ is an almost holomorphic Lagrangian fibration.
\end{custom}


\subsection*{Acknowledgements}
We are grateful to  Brendan Hassett, Daniel Huybrechts, J\'anos Koll\'ar, Alex K\"uronya, Manfred Lehn, Eyal Markman, James McKernan, Keiji Oguiso, Nicolas Perrin, and Chenyang Xu for interesting discussions and suggestions. Moreover, we thank Daniel Barlet, Akira Fujiki, Martin Gulbrandsen, Alan Huckleberry, Matei Toma, and Claire Voisin for answering our questions via email.

The support of the DFG
through the SFB/TR 45 ``Periods, moduli spaces and arithmetic of algebraic varieties'', Forschergruppe 790 ``Classification of Algebraic Surfaces and Compact Complex Manifolds'', and the third author's Emmy--Noether
project was invaluable for the success of the collaboration. During the preparation of the paper, the first author enjoyed the hospitality of the Mathematics Department at Princeton University. He
gratefully acknowledges the support of the
Baden--W\"urttemberg--Stiftung through the ``Eliteprogramm f\"ur
Postdoktorandinnen und Postdoktoranden''. The third author was also partly
supported by the Hausdorff Center for Mathematics in Bonn.

\section{Preliminaries on hyperk\"ahler manifolds}
We collect a few basic definitions and properties of the objects of our study.
\begin{defin}
An \emph{irreducible holomorphic symplectic manifold} or \emph{hyperk\"ahler manifold} is a simply--connected compact K\"ahler manifold $X$ such that $H^0\bigl(X,\, \Omega_X^2 \bigr)$ is spanned by an everywhere non--degenerate holomorphic two--form $\sigma$.
\end{defin}

Actually, the notion of hyperk\"ahler manifold is of differential--geometric origin and stands for a Ricci--flat K\"ahler manifold with holonomy group $\Sp(n)$. It was shown by Beauville in \cite[Prop 4]{beauville83} that this condition is equivalent to the existence of a holomorphic symplectic form unique up to scalars; often the terms \emph{irreducible holomorphic symplectic manifold} and \emph{hyperk\"ahler manifold} are therefore used synonymously.

\subsection{The Beauville--Bogomolov form}\label{sect:q}
The second cohomology $H^2\bigl(X,\, \mathbb{Z}\bigr)$ of a hyperk\"ahler manifold $X$ carries a natural, integral symmetric bilinear form
\[q = q_X \colon H^2\bigl(X,\, \mathbb{Z}\bigr) \times H^2\bigl(X,\, \mathbb{Z}\bigr) \to \mathbb{Z},\]
the so--called \emph{Beauville--Bogomolov--Fujiki form} (see \cite[Thm.~5]{beauville83} or \cite[Def.~22.8]{HuybrechtsHKM}). Since we need to consider the restriction of this form to subspaces where it might be degenerate, we give its signature as a triple containing (in this order) the number of positive, zero, and negative eigenvalues of the associated real symmetric bilinear form. In this notation $q$ has  signature $(3,0, b_2(X) -3)$, and its restriction to  $H^{1,1}(X, \mathbb R)$  has signature $(1, 0, h^{1,1}-1)$, see~\cite[Cor.~23.11]{HuybrechtsHKM}.

Let $\rho = \rho(X)$ be the Picard number of $X$, that is, the rank of the N\'eron--Severi group $\mathrm{NS}(X)=H^{1,1}(X)\cap H^2(X, \IQ)$.  We distinguish hyperk\"ahler manifolds according to the signature of the restriction of $q$ to $\mathrm{NS}(X)$. We call $X$ \emph{hyperbolic} if $q\restr{\mathrm{NS}(X)}$ has signature $(1, 0 , \rho -1)$, \emph{parabolic} if $q\restr{\mathrm{NS}(X)}$ has signature $(0, 1 , \rho -1)$, and \emph{elliptic} if $q\restr{\mathrm{NS}(X)}$ has signature $(0, 0 , \rho)$.
The relevance of these notions is underlined by the following
 result of Huybrechts.
\begin{theo}[Prop.~26.13 of \cite{HuybrechtsHKM}]
A hyperk\"ahler manifold $X$ is projective if and only if $X$ is hyperbolic.
\end{theo}

\subsection{Lagrangian fibrations}\label{sect: lagrfibr}
\begin{defin}
An $n$--dimensional analytic subvariety $Z$ of a $2n$--di\-men\-sio\-nal hyperk\"ahler manifold $X$ is called \emph{Lagrangian} if the holomorphic symplectic form $\sigma$ restricts to a trivial two--form on the smooth part of $Z$.
\end{defin}
As a consequence of Kodaira's embedding theorem, Lagrangian submanifolds turn out to be projective, independent of the projectivity of the ambient hyperk\"ahler manifold.
\begin{prop}[Prop.~2.1 of \cite{campana06}]\label{prop:Lagrangianprojective}
Let $X$ be a hyperk\"ahler manifold and $L \subset X$ a Lagrangian submanifold. Then $L$ is projective.
\end{prop}
\begin{defin}
Let $X$ be a hyperk\"ahler manifold. A \emph{Lagrangian fibration} on $X$ is a holomorphic map $f\colon X \to B$ with connected fibres onto a normal complex space $B$ such that every irreducible component of the reduction of every fibre of $f$ is a Lagrangian subvariety of $X$.
\end{defin}
Due to work of Matsushita one knows that Lagrangian fibrations are the only non--trivial maps from hyperk\"ahler manifolds to lower--dimensional spaces:
\begin{theo}[\cite{matsushita99, matsushita00,matsushita01, matsushita03}]\label{thm: matsushita}
Let $X$ be a hyperk\"ahler manifold of dimension $2n$. If $f\colon  X \to B$ is a holomorphic map with connected fibres onto a normal complex space $B$ with $0 < \dim B < \dim X$, then $f$ is a Lagrangian fibration. In particular, $f$ is equidimensional and $\dim B = n$. Furthermore, every smooth fibre of $f$ is a complex torus.
\end{theo}
In fact, using Proposition~\ref{prop:Lagrangianprojective} we see that every smooth fibre of $f$ is projective, that is, an abelian variety.

\subsection{Almost holomorphic Lagrangian fibrations}
Let $X$ be a hyperk\"ahler manifold and let $f\colon X \dashrightarrow B$ be a dominant meromorphic map. Recall that $f$ is called \emph{almost holomorphic} if there is a (Zariski) open subset $U\subset B$ such that the restriction $f\restr{\inverse f (U)}\colon\inverse f (U) \to U$ is holomorphic and proper. A \emph{strong fibre} of an almost holomorphic map $f$ is a fibre of $f\restr{\inverse f (U)}$.

\begin{defin}
A dominant meromorphic map $f\colon  X \dashrightarrow B$ is called \emph{almost holomorphic Lagrangian fibration} if  $f$ is almost holomorphic with connected fibres and the reduction of every irreducible  component of a fibre of $f\restr{\inverse f (U)}$ is a Lagrangian subvariety of $X$.
\end{defin}
If $f$ is not almost holomorphic, the na\"{i}ve notion of Lagrangian fibre is not too well behaved. In particular, there might not be a single fibre which is isomorphic to a complex torus. For an explicit example consider a pencil of higher--genus curves on a K3--surface.

\begin{rem}\label{rem: pullback}
Recall that if $A$ is a divisor on $B$, then its pullback via $f$ is defined either geometrically as the closure of the pullback on the locus where $f$ is holomorphic, or on the level of locally free sheaves as $f^*\ko_B(A):=(p_*\widetilde f^* \ko_B(A))^{\vee\vee}$,  where $p\colon\widetilde X\to X$ is a resolution of indeterminacies leading to a diagram
\begin{equation*}
\begin{gathered}
\xymatrix{ &\widetilde X\ar[dl]_p \ar[dr]^{\widetilde f}\\ X\ar@{-->}[rr]^f && B,}
\end{gathered}
\end{equation*}
see also \cite[Setup 3.2]{cop}.
\end{rem}

\section{Preliminaries on Barlet spaces}
The Barlet space of a complex space $X$ is the complex--analytic analogue of the Chow variety of a projective algebraic variety. It parametrises compact cycles (with multiplicity) in  $X$. We shortly recall the relevant definitions as well as important criteria to recognise analytic families. The authoritative reference on the subject is \cite{barlet75}, a survey of the circle of ideas surrounding this fundamental construction can be found in \cite{SCVVII}.

In this section, all complex spaces are assumed to be reduced.

\subsection{Basic definitions and properties }

\begin{defin}
Let $X$ be a complex space and $n \in \mathbb{N}$ an integer. An \emph{$n$--cycle} in
$X$ is a finite linear combination $Z = \sum_{i \in I} k_i Z_i$ where $k_i \in \IN^{>0}$ and the $Z_i$'s are pairwise distinct, irreducible, $n$--dimensional compact analytic subsets of $X$. The \emph{support} of $Z$, denoted $\mathrm{supp}(Z)$, is the union of the $Z_i$'s. The set of all $n$--cycles in $X$ is denoted by $\gothB_n(X)$;  the set of all cycles $\gothB(X)=\bigcup_{n\in \IN}\gothB_n(X)$ is called the \emph{Barlet space} of $X$.
\end{defin}
In \cite{barlet75} a natural complex structure is constructed on $\gothB(X)$. With the appropriate definition of \emph{analytic family of compact $n$--cycles}, see \cite[d\'efinition fondamentale, p.~33]{barlet75}, Barlet  showed that $\gothB_n(X)$ represents the functor
\begin{align*}
\gothF ^n_X :&\text{ (reduced complex spaces)} \to \text{(sets)}  \\
&S \mapsto \{\text{analytic families of compact $n$--cycles parametrised by }S\}
\end{align*}
In particular, there exists a universal family $(Z_s)_{s\in \gothB_n(X)}$ of compact $n$--cycles in $X$ para\-metrised by $\gothB_n(X)$. Furthermore, if $(Z_s)_{s\in S}$ is an analytic family of compact $n$--cycles in $X$, then there exists a holomorphic map $\mu\colon S \to \gothB_n(X)$ inducing $(Z_s)_{s\in S}$, the so--called \emph{classifying map}.

We are not going to use the definition of analytic family explicitly, but rather apply the following useful criterion, see also \cite[Ch.~I, \S 2, Thm.~1]{barlet75}:
\begin{prop}[Thm.~2 of \cite{barlet99}]\label{prop:analyticfamiliycriterion}
Let $X$ be a complex manifold and $S$ a normal complex space. Let $\Gamma \subset S \times X$ be an analytic subset which is proper and equidimensional over $S$ with purely $n$--dimensional fibres. Then, there exists a unique analytic family of compact $n$--cycles $(Z_s)_{s\in S}$ parametrised by $S$ satisfying the following conditions:
\begin{enumerate}
\item For general $s \in S$, we have $Z_s = \mathrm{supp}(Z_s)$; in other words, all multiplicities are equal to $1$.
\item For all $s \in S$, we have $\{s\} \times \mathrm{supp}(Z_s) = \Gamma \cap (\{s\} \times X)$ as sets.
\end{enumerate}
\end{prop}
Sometimes it is useful to discuss families of cycles where the dimension is allowed to vary. This leads to the concept of a \emph{meromorphic family of cycles}, which is an analytic family of cycles on a dense open set such that its classifying map is meromorphic. The following result relates this notion to geometry:
\begin{prop}[Prop.~2.20 in Ch.~VIII of \cite{SCVVII}]\label{prop:meromofamily}
Let $X$ and $S$ be irreducible complex spaces, $\dim S = d$. There exists a natural identification between
\begin{enumerate}
\item meromorphic maps $\mu\colon S \dashrightarrow \gothB_n(X)$, and
\item $S$--proper purely $(n+d)$--dimensional cycles $\Gamma$ in $X \times S$,
\end{enumerate}
given by considering the graph of the meromorphic map $\mu$.
Accordingly, we call $\Gamma$ the \emph{graph of the meromorphic family} determined by $\mu$.
\end{prop}
A drawback of the Barlet space is that it does not lend itself to infinitesimal computations; to understand its local structure we need to relate it to the  the Douady space $\mathscr{D}(X)$, which is the complex--analytic analogue of the Hilbert scheme (see for example \cite{douady66} or \cite[Ch.~VIII]{SCVVII}). The local structure of the Douady space can be studied via deformation theory of submanifolds in $X$, and the following comparison result then allows to obtain local information about the corresponding Barlet space.
\begin{prop}[petit th\'eor\`eme in Ch.~V, \S 3 of \cite{barlet75})]\label{prop:DouadyBarlet}
Let $X$ be a complex manifold, and let $Z$ be a compact submanifold of $X$. Let $\mathscr{D}_n(X)$ be the (open and closed) subspace of the Douady space $\mathscr{D}(X)$ that parametrises ideals with purely $n$--di\-men\-sio\-nal support in $X$. Then, the natural holomorphic map $\mathscr{D}_n(X)_{\rm red} \to \gothB_n(X)$ is locally biholomorphic at $[Z] \in \mathscr{D}_n(X)_{\rm red}$.
\end{prop}

The following fundamental result was proven by Lieberman and Fujiki. It often allows to argue along similar lines as in an algebraic setup.
\begin{theo}[\cite{lieberman}, \cite{fujiki78b}]\label{thm:Barletcompact}
Let $X$ be a compact K\"ahler space. Then, every connected component of $\gothB(X)$ is compact.
\end{theo}
It is folklore that this result can be extended to  the relative case.  Since a reference was hard to track down, for the convenience of the reader we give a short argument inspired by \cite[Prop.~1.1]{TelemanDloussky}.
\begin{prop}\label{prop: folklore}
 Let $h\colon \gothX\to T$ be a smooth family of compact complex manifolds over a smooth space $T$. If some fibre $\gothX_0$ is K\"ahler, then there is an open neighbourhood $U$ of $0$ in $T$ such that each connected component of the relative Barlet space $\gothB(\inverse h(U)/U)$ is proper over $U$.
\end{prop}
\begin{proof}
For simplicity, we may assume that $T$ is a polydisc with centre $0$. Then $T$ does not contain any compact cycles except points, and hence the relative Barlet space $\gothB(\gothX/T)$ and the absolute Barlet space $\gothB(\gothX)$ coincide near any cycle which is associated with a smooth manifold completely contained in one fibre. In order to apply a result of Barlet \cite[Thm.~1]{barlet99} that ensures properness we need to bound the volume of cycles with respect to a suitable hermitian metric, which we will now construct.

By \cite[Thm. 15]{kodaira-spencer60} we can find a smaller polydisc $U\subset T$ containing $0$ such that there exists a real smooth 2--form  $\omega_1$ on $\gothX_U:=\inverse h (U)$ which restricts to a K\"ahler form on each fibre. Let $\omega_2$ be the 2--form on $U$ associated to an  arbitrary hermitian metric on $U$. Then, possibly replacing $U$ with a relatively compact open subset,  there exists a constant $M$ such that $\omega=\omega_1+M\cdot h^*(\omega_2)$ is an everywhere positive 2--form on  $\gothX_U$ that additionally restricts to a K\"ahler form on each fibre. (Note that $\omega$ need not be closed on the total space.) Then, the volume function is defined as
\[ \mathrm{vol}_\omega\colon \gothB(\gothX_U) \to \IR, \qquad [C]\mapsto \int_C\omega^n.\]

Choosing a differentiable trivialisation $\gothX_U\isom \gothX_0\times U$ we identify the real (co)\-ho\-mology  of $\gothX_U$ with the (co)\-ho\-mology of $\gothX_0$. With this identification $\omega$ induces for each $n$ a family of classes $[\omega^n_t]\in H^{2n}(\gothX_0, \IR)$ depending on $t\in U$.

Now let $g\colon\gothB\to U$ be a connected component of $\gothB_n(\gothX_U/U)$ together with the projection to $U$. Every cycle in $\gothB$ induces the same homology class $\alpha=\alpha_\gothB\in H_{2n}(\gothX_0, \IR)$. Consequently, the volume of a cycle  $[C]\in \gothB$ with $h([C])=t$
can be expressed in terms of the pairing,
\[ \mathrm{vol}_\omega(C)=\int_C \omega^n=\langle [\omega^n_t], \alpha\rangle,\]
which implies that $\mathrm{vol}\restr\gothB=g^*\phi$ for a continuous function $\phi\colon U\to \IR$.

For every relatively compact subset $K\subset U$ every cycle in $\inverse g(K)$ is contained in the compact set $\inverse h (\bar K)\subset \gothX_U$ and its volume is bounded by the maximum of $\phi$ on $\bar K$. Thus $\inverse g(K)$ satisfies the assumptions of  \cite[Thm.~1]{barlet99} and is hence relatively compact in $\gothB(\gothX_U)$. In particular, $g^{-1}(K)$ is proper over $K$.
\end{proof}

\begin{rem}
Fujiki \cite[Rem.~4.3]{fujiki78b} shows that properness may fail for $\gothB(\gothX/T)\to T$ even if $\gothX$ is compact and all fibres of $h\colon \gothX \to T$ are projective.
\end{rem}

\subsection{Barlet spaces and meromorphic fibrations}
Recall that if $(Z_s)_{s\in S}$ is an analytic family of compact $n$--cycles in $X$, then its \emph{graph} $\{(x,s)\in X\times S \mid x\in \mathrm{supp}(Z_s)\}$ is an analytic subset of $X\times S$ by \cite[Ch.~VIII, Thm.~2.7]{SCVVII}. If we equip this analytic subset with the reduced structure we obtain a complex space, which is proper over $S$. If $S=\gothB(X)$ and $(Z_s)_{s\in S}$ is the universal family, we will write
\begin{equation}\label{diagram: whole Barlet}
\begin{gathered}
\xymatrix{ \gothU(X)\ar[r]^\epsilon\ar[d]^\pi & X\\ \gothB(X)}
\end{gathered}
\end{equation}
for the complex space associated to the universal family where $\pi$ and $\epsilon$ are induced by the projections. The following lemma (most parts of which are certainly well--known to experts) will be applied in our study of Lagrangian fibrations.
\begin{lem}\label{Chowiso}
Let $X$ be a compact and connected K\"ahler manifold and $f\colon X\to B$ a surjective map with connected fibres to a normal complex space $B$. Let $U_{\rm pure} \subset B$ be the Zariski--open set over which the fibres are of pure dimension $d=\dim X - \dim B$. Let $B_\sm$ denote the set of smooth points of $B$. Then, the following holds:
\begin{enumerate}
\item The graph of $f$ defines a meromorphic family of cycles in $X$.
  \setcounter{saveenum}{\value{enumi}}
\end{enumerate}
Let $\gothB$ be the union of the irreducible components of the Barlet space $\gothB(X)$ that contain all fibres $[X_b]$ of $f$ over points $b \in U_{\rm pure}$, and let $\pi\colon \gothU\to \gothB$ be the projection of the graph of the universal family over $\gothB$.
\begin{enumerate}
  \setcounter{enumi}{\value{saveenum}}
\item $\gothB(X)$ is smooth at $[X_b]$ for any point $b \in U_{\rm pure}\cap B_\sm$ such that $X_b$ is smooth. Consequently, $\gothB$ is irreducible.
\item The (meromorphic) classifying map $\mu\colon B \dasharrow \gothB$ induces a holomorphic bijection of $U_{\rm pure}$ onto its Zariski--open image $\mu(U_{\rm pure}) \subset \gothB$.
\item The evaluation map $\epsilon\colon\gothU \to X$ is an isomorphism on $\pi^{-1}\left(\mu(U_{\rm pure})\right)$ satisfying $\mu \circ f= \pi\circ  \inverse \epsilon$. In particular, $\epsilon$ is bimeromorphic.
\end{enumerate}
\end{lem}
\begin{proof}
By Proposition~\ref{prop:meromofamily} the graph $\Gamma_f\subset X\times B$ of $f$ is the graph of a meromorphic family of cycles in $X$, which proves \refenum{i}.

We next show \refenum{ii} and \refenum{iii}.
Let $b\in U_{\rm pure }$ be a smooth point of $B$ such that the fibre $X_b$ is smooth as well. By Proposition~\ref{prop:DouadyBarlet} the Barlet space is isomorphic to the reduction of the  Douady space near $[X_b]$, and we can therefore estimate the  dimension of $\gothB(X)$ at $[X_b]$ by a deformation--theoretic computation. In a saturated neighbourhood of $X_b$ the map $f$ is an equidimensional holomorphic map between complex manifolds. Since additionally $X_b$ is smooth, $f$ is a smooth submersion near $X_b$. Consequently, the normal bundle of $X_b$ in $X$ is trivial.
It follows that the tangent space of the Douady space at the point $[X_b]$ has dimension $h^0(X_b, \kn_{X_b/X})=\dim B$.

As a consequence of the previous paragraph we have $\dim_{[X_b]} \gothB \leq \dim B$. On the other hand, the image $\mu(B)$ of the meromorphic classifying morphism $\mu\colon B \dasharrow \gothB$ is an analytic subvariety of $\gothB$, since $B$ is compact (Theorem~\ref{thm:Barletcompact}). Moreover, by Proposition~\ref{prop:analyticfamiliycriterion} the restriction of $\mu$  to  $U_{\rm pure}$ is holomorphic, and the image $\mu(U_{\rm pure})$ is Zariski--open in $\mu(B)$. Since the fibres of $f$ are mutually disjoint, $\mu$ is injective on $U_{\rm pure}$. Therefore, we also have $\dim B \leq \dim_{[X_b]} \gothB$. Since $B$ is irreducible, it follows that $\mu$ maps $B$ onto a single $\dim B$-dimensional irreducible component of $\gothB$. However, it then follows from the dimension--computation of the Zariski tangent space of $\mathscr{D}(X)$ at $[X_b]$ made above and from Proposition~\ref{prop:DouadyBarlet} that the Barlet space is smooth, hence irreducible 
 at $[X_b]$. Consequently, $\gothB$ is irreducible and smooth at $[X_b]$. This shows \refenum{ii} and \refenum{iii}.

In order to prove \refenum{iv}, we look at the diagram
\[
\xymatrix{ \gothU \ar[r]^\epsilon\ar[d]^\pi & X\\ \gothB.}
\]
We note for later reference that $\epsilon$ is proper, since $\gothB$ is compact (Theorem~\ref{thm:Barletcompact}). We now restrict our attention to the open subset $U_{\rm pure}$. By pulling back the graph of the universal family over $\gothB$ to $U_{\rm pure}$ via the holomorphic map $\mu\restr{U_{\rm pure}}$, and denoting $\mu(U_{\rm pure})$ by $\bar U$, we obtain the following diagram
\[ \xymatrix{
(\mu^*\gothU\restr{\bar U})_{\mathrm{red}} \ar[r]^{\hat\mu}\ar[d]^f & \gothU\restr{\bar U} \ar[d]^\pi\ar[r]^>>>>>\epsilon & f^{-1}(U_{\rm pure})\\
U_{\rm pure} \ar[r]^\mu& \bar U,
}
\]
where $\mu^*\gothU\restr{\bar U} =\gothU\restr{\bar U}\times_{\bar U} U_{\rm pure}$.
Comparing $(\mu^*\gothU\restr{\bar U})_{\mathrm{red}}$ and $\Gamma_{f}\restr{U_{\rm pure}}$ inside $f^{-1}(U_{\rm pure})\times U_{\mathrm{pure}}$ we see that these are reduced analytic subsets with the same underlying topological space; hence, they coincide. Identifying $\Gamma_{f}\restr{U_{\rm pure}}$ with $(\mu^*\gothU\restr{\bar U})_{\mathrm{red}}$ we conclude that the composition $\epsilon \circ \hat\mu$ coincides with the projection from the graph $\Gamma_f|_{U_{\mathrm{pure}}}$ to $f^{-1}(U_{\mathrm{pure}})$. It is therefore an isomorphism. In particular, $\epsilon$ is a proper birational map with finite fibres onto a normal space, hence an isomorphism.
\end{proof}

The following result provides an extension of the previous discussion to the case of almost holomorphic maps.
\begin{lem}\label{lem:almosthol}
Let  $f\colon  X \dasharrow B$ be a surjective almost holomorphic map with connected fibres from a connected compact K\"ahler manifold $X$ to a normal complex space $B$. Let $b$ be a smooth point of $B$, assume that $F=X_b$ is a smooth strong fibre of $f$ and that $f$ is equidimensional over a neighbourhood of $b$. Then, there exists a unique irreducible component $\gothB$ of $\gothB(X)$ containing $[F]$. Furthermore, the evaluation map $\epsilon\colon~\gothU \to X$ from the graph of the universal family $\gothU$ over this component to $X$ is bimeromorphic.
\end{lem}
\begin{proof}
By Hironaka there exists a modification $p\colon \widetilde X \to X$ of $X$ such that $\widetilde X$ is smooth, $p$ is a projective morphism, and  $\widetilde f := f \circ p$ is holomorphic; in particular, $\widetilde X$ is K\"ahler. Since $f$ is assumed to be almost holomorphic, there exists a Zariski--open smooth subset $U$ in $B$ such that $f\restr{f^{-1}(U)}$ is holomorphic and proper. As a consequence, $p\colon \widetilde X \to X$ can be chosen in such a way that the set where it is not biholomorphic  is disjoint from $f^{-1}(U)$. Moreover, possibly after shrinking $U$,  we may assume that the fibre $\widetilde X_b$ is smooth for all $b \in U$.
Next, we apply Lemma~\ref{Chowiso} to $\widetilde f$ and obtain a diagram
\[
\xymatrix{ \widetilde \gothU \ar[r]^{\widetilde \epsilon}\ar[d]^{\widetilde \pi} & \widetilde X\\ \widetilde \gothB,}
\]
in which $\widetilde \epsilon$ is bimeromorphic.

The graph $\widetilde \gothU \subset \widetilde \gothB \times \widetilde X$ is a pure--dimensional $\widetilde \gothB$--proper analytic subset of $\widetilde \gothB \times \widetilde X$. Mapping it to $\widetilde \gothB \times X$ using the map $p$, we obtain a pure--dimensional $\widetilde \gothB$--proper analytic subset $\Gamma$ of $\widetilde \gothB \times X$. Proposition~\ref{prop:meromofamily} implies that $\Gamma$ is the graph of a meromorphic family of cycles in $X$, parametrised by $\widetilde \gothB$. We denote the corresponding meromorphic map from $\widetilde \gothB$ to $\gothB(X)$ by $q$. Note that the restriction of $q$ to $\mu_{\widetilde f}(U)$ is an isomorphism onto its image. As in the proof of Lemma~\ref{Chowiso}, a dimension--computation of the Zariski tangent space at $[F]$ now shows that $\gothB(X)$ is actually smooth of dimension $\dim \widetilde\gothB = \dim B$ at $[F]$. Thus, denoting by $\gothB$ the  unique irreducible component of $\gothB(X)$ containing the point $[F]$, we see that $q$ is a bimeromorphic map from $\widetilde \gothB$ to $\gothB$. Denoting  the induced bimeromorphic map between the graphs of the universal families by $\widetilde p$, we obtain the following diagram:
\[ \xymatrix{\widetilde\gothU\ar@{-->}[dr]_{\widetilde p} \ar[r]^{\widetilde \epsilon}\ar[dd]_{\widetilde \pi} & \widetilde X\ar[dr]_p\ar[drr]^{\widetilde f}\\
 & \gothU\ar[r]_\epsilon\ar[d]^\pi & X \ar@{-->}[r]_f& B.\\
\widetilde\gothB\ar@{-->}[r]^{q}& \gothB}
\]
Consequently, $\epsilon^{-1}=\widetilde p\circ\inverse{\widetilde\epsilon}\circ\inverse p$ is a meromorphic inverse to $\epsilon$, as claimed.
\end{proof}

\section{The family of deformations of a Lagrangian subtorus}\label{sect:deformations}
Let $X$ be a hyperk\"ahler manifold and assume that $X$ contains a \emph{Lagrangian subtorus}, that is, a smooth Lagrangian subvariety $L$ of $X$ that is biholomorphic to a complex torus.
We want to test if $L$ behaves as if it were the fibre of a fibration.

More precisely, consider an irreducible component $\gothB$ of the Barlet--space containing $[L]$, which will be shown to be unique below, and let $\gothU$ be the graph of the universal family over $\gothB$ as in \eqref{diagram: whole Barlet}. The natural evaluation map $\epsilon \colon \gothU \to X$, which restricts to an embedding on each cycle, is induced by the projection to $X$.
As the general cycle parametrised by $\gothB$ is smooth, there is a proper analytic subset $\Delta \subseteq \gothB$ parametrising singular cycles, which we call the \emph{discriminant locus}; the family of cycles over it will be denoted by $(\gothU \times_\gothB \Delta)_\red=:\gothU_\Delta$. We will constantly refer back to this setup which we summarise in the following diagram
\begin{equation}\label{diagram:Barlet}
\begin{gathered}
\xymatrix{ \gothU_\Delta \ar@^{(->}[r]\ar[d]& \gothU \ar[r]^\epsilon\ar[d]_\pi & X\\ \Delta \ar@^{(->}[r]& \gothB.}
\end{gathered}
\end{equation}
By Lemma~\ref{Chowiso} the torus $L$ is the fibre of a fibration if and only if $\epsilon$ is an isomorphism (see Lemma~\ref{lem:epsilon birational} below for a sharpening of this observation). Evidence that this has a chance to be true can be obtained by studying deformations of $L$ in $X$.
\begin{lem}\label{deflem} Let $X$ be a hyperk\"ahler manifold of dimension $2n$ and let $L$ be a Lagrangian subtorus of $X$. Then, the following holds.
\begin{enumerate}
 \item The Barlet space $\gothB(X)$ is smooth of dimension $n$ near $[L]$. In particular, $[L]$ is contained in a unique irreducible component $\gothB$  of $\gothB(X)$ and  $\gothU$ is smooth of dimension $2n$ near $\pi^{-1}([L])$.
\item If $[L'] \in \gothB$ represents a smooth subvariety $L'$, then $L'$ is a Lagrangian subtorus of $X$.
\item The morphism $\epsilon$ is finite \'etale along smooth fibres of $\pi$. In particular, a sufficiently small deformation of $L$ is disjoint from $L$ and, there are no positive--dimensional families of smooth fibres through a general point $x\in X$.
\end{enumerate}
\end{lem}
\begin{proof}
We first consider the Douady space $\mathscr{D}(X)$ near $[L]$. The proof of \cite[Thm. 8.7 (ii)]{DonagiMarkman} works equally well in the K\"ahler setting (see also \cite[Thm.~2.2]{Ra92} for the statement, and \cite[Thm.~VI.6.1]{LehnDiss} for a detailed proof), so $\mathscr{D}(X)$ is smooth at $[L]$ with tangent space $H^0\bigl(L,\, N_{L/X}\bigr)$.
Since $L$ is Lagrangian, the symplectic form induces an isomorphism $N_{L/X}\isom \Omega^1_L$. Since moreover $L$ is a complex torus,  we compute
\[  \dim_{[L]}\mathscr D (X) = h^0\bigl(L,\, N_{L/X}\bigr) = h^0\bigl(L,\, \Omega^1_L\bigr) = h^0\bigl(L,\, \ko_L^{\oplus n}\bigr) = n.\]
The comparison between Douady and Barlet spaces (Proposition~\ref{prop:DouadyBarlet}) then implies that also $\gothB(X)$ is smooth of dimension $n$ at $[L]$. This proves \refenum{i}.

Item \refenum{ii} is  \cite[Thm.~8.7 (i)]{DonagiMarkman}; it also follows from the proof of \cite[Lem.~1.5]{voisin92}.

For \refenum{iii}, let $y \in \pi^{-1}([L])$ with smooth $L$ and $\epsilon(y)=x$. As $\gothU$ and $X$ are both smooth at $y$ and $x$, respectively, and since $\epsilon$ is proper (Theorem~\ref{thm:Barletcompact}), it suffices to show that $T_{\gothU}(y)\to T_{X}(x)$ is an isomorphism. We have already noted that near the smooth point $[L] \in \gothB(X)$ the Barlet--space is biholomorphic to the Douady space. In addition, the graph $\gothU$ of the universal family over $\gothB$ is biholomorphic to the universal family over $\mathscr{D}(X)$ near the fibre $\inverse\pi([L])$. Consequently, the universal family over the Douady space can also be interpreted as the incidence variety
\[\left\{(z,[L]):z\in L\subseteq X\right\} \subset X \times \mathscr{D}(X).\]
On the level of tangent spaces this interpretation leads to an exact sequence \begin{equation}\label{eq:univfamilytangent}
0 \to T_{\gothU}(y) \to T_{X}(x)\oplus H^0\bigl(L,\,N_{L/X}\bigr) \to N_{L/X}(x)
\end{equation}
and the composition of the first morphism with the projection to $T_{X}(x)$ is the differential of $\epsilon$. Now the  Lagrange condition on $L$ implies that the horizontal arrows in the following diagram are isomorphisms
\[\begin{xymatrix}{
H^0\bigl(L,\, N_{L/X}\bigr) \ar^{\cong}[r] \ar_{\mathrm{ev}}[d]& H^0\bigl(L,\, \Omega_L^1 \bigr) \ar^{\mathrm{ev}}[d] \\
N_{L/X} (x) \ar^{\cong}[r]& \Omega_L^1(x).
}
\end{xymatrix}\]
Since $L$ is a torus, the cotangent bundle $\Omega_L^1$ is trivial, and therefore the evaluation map $\mathrm{ev}\colon H^0\bigl(L,\, \Omega_L^1 \bigr) \to \Omega_L^1(x)$ is an isomorphism. It follows that the evaluation map $\mathrm{ev}\colon H^0\bigl(L,\,N_{L/X}\bigr) \to N_{L/X}(x)$ is likewise an isomorphism. Using \eqref{eq:univfamilytangent} we conclude that the same is true for $T_{\gothU}(y)\to T_{X}(x)$.
\end{proof}
\begin{lem}\label{lem:epsilon birational}
Let $X$ be a hyperk\"ahler manifold containing a Lagrangian subtorus $L$. Then $X$ admits an almost holomorphic Lagrangian fibration with strong fibre $L$ if and only if the evaluation map $\epsilon$ in diagram \eqref{diagram:Barlet} is bimeromorphic.
\end{lem}
\begin{proof}
If $\epsilon$ is bimeromorphic, then the Stein factorisation of a holomorphic model of $\pi\circ\inverse\epsilon$ yields a meromorphic map from $X$ to a compact normal complex space. This map is a holomorphic Lagrangian fibration near $L$ because $\epsilon$ is \'etale near $L$ by Lemma~\ref{deflem}.
The other direction follows from Lemma~\ref{lem:almosthol}.
\end{proof}

In view of Lemma~\ref{deflem} one might wonder whether, given a $k$-cycle $L$ on a compact K\"ahler manifold $X$ of dimension $n+k$ such that
\begin{itemize}
\item the Barlet space of $X$ is smooth at $[L]$,
\item the irreducible component containing $L$ is $n$--dimensional, and
\item the evaluation map $\epsilon$ is generically \'etale,
\end{itemize}
the map $\epsilon$ automatically has degree one. In other words, are deformations of $L$ fibres of a (meromorphic) fibration on $X$?
As the following example shows this is in general not the case.
\begin{exam}
Let $X = \{f_3 = 0\}\subset \mathbb{P}^4$ be a smooth cubic threefold. Then, $X$ is a Fano manifold covered by lines. Let $L \subset X$ be a general line in the covering family $\gothF$ of lines. Then, the normal bundle of $L$ is trivial, and the Barlet space $\gothB (X)$ of $X$ is smooth of dimension $2$ at the point $[L]$. Let $\gothB$ be the irreducible component containing $ [L] $. Then, the evaluation morphism from the graph $\gothU$ of the universal family over $\gothB$ to $X$ is \'etale along smooth fibres of $\pi\colon  \gothU \to \gothB$; i.e., properties \refenum{i} and \refenum{ii} of Lemma~\ref{deflem} hold. However, an explicit computation (see for example \cite[Sect.~1.4.2]{HwangVMRTSurvey}) shows that the variety of tangents to the family $\gothF$ at a general point of $X$ consists of $6$ points. It follows that $\epsilon$ has degree $6$ and is therefore not birational.
\end{exam}

\section{The non--projective case}\label{sect: non-proj}
In this section we answer Question B positively in the non--projective case.
\begin{theo}\label{mainnonproj}
Let $X$ be a non--projective hyperk\"ahler manifold of dimension $2n$ containing a Lagrangian subtorus $L$. Then the algebraic dimension of $X$ is $n$, and there exists an algebraic reduction $f\colon  X \to B$ of $X$ that is a holomorphic Lagrangian fibration with fibre $L$.
 \end{theo}
As an immediate consequence we obtain:
\begin{cor}\label{cor: almost holomorphic non projective}
Let $X$ be a non--projective hyperk\"ahler manifold and $f\colon X \dasharrow B$ an almost--holomor\-phic Lagrangian fibration. Then, there exists a holomorphic model for $f$; that is, there exists a normal complex space $B'$, bimeromorphic to $B$, such that $f$ extends to a holomorphic Lagrangian fibration $f\colon X \to B'$.
\end{cor}
In Section~\ref{sect:holmodels} we will prove a corresponding statement in the projective setup using the minimal model program.
\begin{proof}
The general strong fibre of an almost--holomorphic Lagrangian fibration is a Lagrangian torus, so $X$ satisfies the assumptions of Theorem~\ref{mainnonproj}. Consequently, its algebraic reduction can be chosen to be the desired holomorphic fibration.
\end{proof}
The  proof of Theorem \ref{mainnonproj} relies heavily on the results of \cite{cop} where the algebraic reduction of a non--algebraic hyperk\"ahler manifold is studied in detail; we are very grateful to K. Oguiso for pointing us to this reference.

\subsection{Covering families of subvarieties}
One of the guiding ideas of the article \cite{cop} is to study the existence of subvarieties in fibres of algebraic reductions. The following definition collects basic notions related to this general problem.
\begin{defin}
Let $X$ be a compact K\"ahler manifold.
\begin{enumerate}
 \item We call $X$ \emph{simple} if it is not covered by positive--dimensional
irreducible compact proper analytic subsets.
\item We call $X$ \emph{isotypically semisimple} if there
exists a simple compact K\"ahler manifold $S$, a natural number $k \geq 1$, a complex space $Y$, and
generically finite surjective holomorphic maps $p\colon Y \to X$ and $q\colon Y \to S^k$.
\end{enumerate}
\end{defin}
\begin{rem}
If $X$ is simple of dimension $\dim(X) \geq 2$, then $a(X) =0$; see \cite[Ch.~VIII, Rem.~3.40]{SCVVII}.
\end{rem}
The following result will exclude some manifolds from being isotypically semisimple. It forms the technical core of our argument to prove Theorem~\ref{mainnonproj}.
\begin{prop}\label{prop:notsemisimple}
Let $X$ be a reduced complex space having a covering analytic family $Z \subseteq X \times T$ of positive dimensional 
 subspaces $Z_t \subsetneq X$ parametrised by a  compact complex space $T$. If there is  a dense open subset $U\subseteq T$ such that $Z_t$ is an irreducible Moishezon space for $t \in U$, then $X$ is not isotypically semisimple with $a(X) = 0$.
\end{prop}
Before we proceed to the proof we need a preliminary Lemma, similar to \cite[\S 3, Prop.~2]{fujiki82}.
\begin{lem}\label{lem:irred covering}
Let $X$, $T$ be reduced compact complex spaces and $W \subseteq X \times T$ a reduced and  irreducible analytic subspace such that $\pi\colon W\to T$ is surjective. Then there exists a commutative diagram
\[
\xymatrix{
W'\ar@^{(->}[r] \ar[d]& X\times T' \ar^{\id \times h}[d]\\
W\ar@^{(->}[r] & X\times T
 }
\]
where $W'$ is a reduced complex space, $T'$ is a normal irreducible  complex space,
 and finite surjective morphisms $h\colon T'\to T$ and $h'\colon W'\to W$ such that the generic fibre of $W'\to T'$ is irreducible.
\end{lem}
\begin{proof}
Take a resolution of singularities $\widetilde{W} \to W$. As $W$ is irreducible, so is $\widetilde{W}$. By generic smoothness, there is a dense open $V\subseteq T$ such that $U:=\pi^{-1}(V)\to V$ is smooth. So if $\pi\vert_U$ has reducible fibres, it has to have disconnected fibres. Let $\widetilde{W}\to T' \to T$ be the Stein factorisation of $\widetilde{W}\to T$. Note that $T'$ is irreducible, as $\widetilde{W}$ is irreducible, and normal, as it is a Stein factorisation and $\widetilde{W}$ is normal. Clearly, $\widetilde{\pi}\colon\widetilde{W}\to T'$ is also generically smooth. Moreover, $\widetilde{\pi}$ has connected fibres, hence the smooth fibres are irreducible. We now obtain $W' \subseteq X\times T'$ with the required properties as the image of $\widetilde{W}\to X\times T'$.
\end{proof}

\begin{rem}\label{rem:irred covering}
The proof of Lemma~\ref{lem:irred covering} shows that generically the fibres of $W'\to T'$ are just the irreducible components of the fibres of $W\to T$. More precisely, there is a dense open set $T^0\subseteq T$ such that for each $t\in T^0$ we have $W_t=\bigcup_{t'\in h^{-1}(t)} W'_{t'}$ as subspaces in $X$.
\end{rem}
\begin{proof}[Proof of Proposition~\ref{prop:notsemisimple}]
By definition $T$ is irreducible, and by pulling back the family to the normalisation if necessary, we may also assume $T$ to be normal. As $X$ is irreducible, there has to be an irreducible component of $Z$ dominating $X$, so we may additionally assume $Z$ to be irreducible.

Contrary to our claim, suppose that $X$ is  isotypically semisimple of algebraic dimension 0. Then, there is a simple compact K\"ahler manifold $S$, a compact complex space $Y$, and  generically finite surjective holomorphic maps
\[
\begin{xymatrix}{
Y \ar^p[r] \ar_{q}[d]& X \\
S \times \dots \times S & .
 }\end{xymatrix}
\]
We may replace $Y$ by any $Y'$ that maps generically finite onto $Y$. Hence, by resolving singularities we may assume $Y$ to be smooth, and by \cite[Ch.~VII, Thm.~2.8]{SCVVII} we may assume $p\colon Y\to X$ to be projective.
We now want to derive a contradiction by constructing from $Z$ a covering family of non--trivial cycles on the simple manifold $S$.

Consider an irreducible component $Z'\subseteq Y\times T$ of $(p \times id_T)^{-1}(Z)$ dominating $S^k$ and map it to $S^k \times
T$ via $(q \times id_T)$ to obtain the graph of a meromorphic family $W \subseteq S^k \times T$ parametrised by $T$ and dominating $S^k$, see Proposition~\ref{prop:meromofamily}.
Since $p$ is projective, for every $t\in U$ the fibre $Z'_t$ is Moishezon. Because images of
Moishezon spaces are Moishezon, see for example \cite[Ch.~VIII, Cor.~2.24]{SCVVII}, $W_t$ is likewise Moishezon.

By Lemma~\ref{lem:irred covering} and Remark~\ref{rem:irred covering}, we may replace  $T$ by a normal space (which we again denote by $T$) such that, after possibly shrinking $U$, the fibre $W_t$ is an irreducible Moishezon space for all $t\in U$.

Let $p_i\colon S^k\to S$ be a projection to one of the factors such that the general $W_t$ is not mapped to a point. Here, we use that $Z_t$, hence $W_t$ is positive dimensional. Then, $\overline{W}=(p_i\times\id_T)(W)$ yields a meromorphic covering family of $S$ with generically irreducible fibres. Moreover, $\overline{W}_U=\overline{W}\times_T U \subseteq \overline{W}$ is dense, dominates $S$, and $\overline{W}_U\to U$ has Moishezon fibres.

By semicontinuity of the fibre dimension there is a dense open subset $V\subset U$ such that $\overline W_t$ is of pure dimension $d$, which is the minimal dimension of a fibre of $\overline W_U\to U$; we have chosen the projection  $p_i\colon S^k\to S$ such that $ d > 0$.

If $d=\dim S$, then $S=\overline W_t$ for all $t$, because $S$ is irreducible and $\overline W_t$ is a closed  subspace of the same dimension. But then $S$ is Moishezon, hence $S^k$ is Moishezon, and so are $Y$ and $X$, as $q$ is generically finite and $p$ is surjective. This however contradicts $a(X)=0$.

If $0<d<\dim S$, consider the meromorphic classifying map $\mu\colon T \dasharrow \gothB_d(S)$, where $\gothB_d(S)$ is the Barlet space of $S$ classifying families of compact analytic $d$--cycles. It follows that the image $\mu(T)$ parametrises an analytic family of positive--dimensional cycles covering $S$, which contradicts the assumption that $S$ is simple.

Since by construction $d>0$ we reach the conclusion that $X$ cannot be isotypically semisimple with $a(X)=0$, as claimed.
\end{proof}

\subsection{Proof of Theorem~\ref{mainnonproj}}
Using the preparatory results obtained above we are now in the position to prove the main result of this section.

Suppose first that $a(X) = 0$. Then, $X$ is isotypically
semisimple by \cite[Cor.~2.5 (2)]{cop}. However, the deformations of $L$ cover $X$ and smooth deformations of $L$ in $X$ are projective (see Proposition~\ref{prop:Lagrangianprojective}). We may hence apply Proposition~\ref{prop:notsemisimple} to arrive at a contradiction.

As $X$ is a non--projective K\"ahler manifold, it cannot be Moishezon \cite{moishezon67}, so $0<a(X)<2n$. Then by  \cite[Thm.~3.1~(2)]{cop} the manifold $X$ is parabolic in the sense of Section~\ref{sect:q}. Consequently, by Theorem 2.3 and Theorem 2.4 of \cite{cop}, one can choose an algebraic reduction of one of the following two forms:
\begin{enumerate}
\item $f\colon X \to B$ is a holomorphic Lagrangian fibration, in particular, $a(X) =n$.
\item $f\colon X \dasharrow B$ is not almost holomorphic, and the very general fibre $X_b$ ($b \in B$) is
    isotypically semisimple with $a(X_b) = 0$. Moreover, in this case $a(X) < n$ \cite[Thm.~3.6, Thm.~3.7]{cop}.
\end{enumerate}

Let us first exclude the case \refenum{ii}. Assuming that \refenum{ii} holds we will construct a  family of positive-dimensional, generically projective cycles covering $X_b$. Consider a very general fibre $X_b$ of the algebraic reduction that intersects a general deformation $L_t$ of $L$, with $t\in \gothB$. Then, by the last statement in \refenum{ii} the family of intersections $X_b\cap L_t$ yields a covering analytic family of \emph{positive--dimensional} generically projective  subvarieties in $X_b$, as follows: In the notation of Diagram \eqref{diagram:Barlet} consider an irreducible component $\gothV$ of $\epsilon^{-1}(X_b) \to \pi(\inverse \epsilon (X_b))$ such that the evaluation morphism $\epsilon_\gothV\colon  \gothV \to X_b$ is still surjective and let $\gothC := \pi(\gothV) \subset \gothB$. Using Lemma~\ref{lem:irred covering} we may assume that the general fibre of $\pi_\gothV\colon  \gothV \to \gothC$ is irreducible. By Proposition~\ref{prop:meromofamily}, there exists a meromorphic map $\mu_\gothV\colon  \gothC \dashrightarrow \gothB(X_b)$ with graph $\gothV$. Then, the universal family over the image $\mu_\gothV(\gothC)$ is the desired family. Its general fibre is isomorphic to (a component of) $X_b \cap L_t \subset L_t$ for some general $t$ and therefore projective. Thus the  assumptions of Proposition~\ref{prop:notsemisimple} are satisfied for $X_b$ together with this family. This contradicts the fact that $X_b$ is isotypically semisimple with $a(X_b) = 0$.

So we are in case \refenum{i}, that is, the algebraic reduction
of $X$ has a holomorphic model $f\colon X\to B$ which is  a Lagrangian fibration. We still need to
show  that $L$ is one of the fibres. By \cite[Thm.~3.4]{cop} the map $f$ is the morphism associated to a line bundle $\kl$ satisfying $c_1(\kl).C=0$ for all curves $C\subseteq X$. It follows that every curve in $X$ is contracted by $f$. As a consequence, if $Y\subseteq X$ is a subvariety such that any two general points of $Y$ can be joined by a curve lying on $Y$, then $f$ contracts $Y$ to a single point in $B$. By Lemma~\ref{deflem} \refenum{iii} and Proposition~\ref{prop:Lagrangianprojective} this applies to $L$ itself and to any smooth deformation $L'$ of $L$. As $L'$ has dimension $n$, it is a component of a fibre of $f$. It follows that the image of the general fibre of $\pi$ under $\epsilon$ is a fibre of $f$. Thus, also $L$ is a fibre of $f$ by Lemma~\ref{Chowiso}~\refenum{iv}. This shows that the algebraic reduction is a holomorphic Lagrangian fibration with fibre $L$.
\qed

\section{Transporting fibrations along deformations}
We would now like to extend the result obtained in the last section to \emph{projective} hyperk\"ahler manifolds $X$ containing a Lagrangian torus $L$. The natural idea is to consider a deformation of  $(X,L)$ to a non--projective pair $(X', L')$ and then try to transport the fibration along the family. While non--projective hyperk\"ahler manifolds are dense in the universal deformation space of $X$, this might a priori no longer be true for deformations of the pair. We start by introducing some terminology.

\begin{defin}\label{def: defo}
 Consider a family of maps
 \begin{equation}\label{eq:defo}
\begin{gathered}
\xymatrix{ \gothL \ar[rr]^j\ar[dr]_p && \gothX\ar[dl]^h\\& T}
 \end{gathered}
\end{equation}
 over a connected complex space $T$. We call this datum a \emph{family of pairs of  a hyperk\"ahler manifold together with a Lagrangian subtorus} if  $p$ is a smooth family of complex tori,  $h$ is a smooth family of hyperk\"ahler manifolds,  and $j$ is a closed embedding, such that $j_t(\gothL_t)$ is a Lagrangian submanifold of $\gothX_t$ for all $t\in T$.

If $X$ is a hyperk\"ahler manifold containing a Lagrangian subtorus $L$ such that for some point $0\in T$ the map $j_0:\gothL_0\to \gothX_0$ is the inclusion of $L$ into $X$, then we call  such a family  of pairs, or by abuse of
terminology any fibre of such a family, a \emph{deformation of the pair $(X,L)$}.
\end{defin}
\begin{rem}
If $(X,L)$ is a hyperk\"ahler manifold together with a Lagrangian sub\-to\-rus, and if a family of maps as in diagram \eqref{eq:defo} is a smooth deformation of the pair $(X,L) = (\gothX_0, \gothL_0)$, then the fibres of $h$ are automatically hyperk\"ahler in an open neighbourhood of $0\in T$. Note however that a deformation in the large of $X$ might even fail to be K\"ahler, so the condition on $h$ is necessary to remain in our framework.

 Since every $\gothX_t$ is a hyperk\"ahler manifold and  $\gothL_0$ is a Lagrangian submanifold, for $t$ near $0$ the submanifold  $j_t(\gothL_t)$ is also Lagrangian, see \cite[Lem 1.5]{voisin92} and the proof of Lemma~\ref{deflem}. Hence, the same is true for all $t\in T$, as $p_*\Omega_{\gothL/T}^2$ is locally free and $T$ is connected.
\end{rem}
The main result of this section is
\begin{theo}\label{notstabproj}
Let $X$ be a projective hyperk\"ahler manifold and $L\subset X$ a Lagrangian subtorus. Then the following are equivalent.
\begin{enumerate}
\item $X$ admits an almost holomorphic Lagrangian fibration with strong fibre $L$.
 \item The pair $(X,L)$ admits a small deformation $(X', L')$ with non-projective $X'$.
\item There exists an effective divisor $D$ on $X$ such that $c_1(\mathscr{O}_X(D)|_L) = 0 \in H^{1,1}\bigl(L, \, \mathbb{R} \bigr)$. 
\end{enumerate}
\end{theo}
 
The proof of Theorem~\ref{notstabproj} will be given at the end of this section. It relies on two technical lemmas: Lemma \ref{lemma:char stab proj} which characterises deformability to non-projective $X'$ in terms of periods and Lemma \ref{lem:closed} which allows to ``transport'' Lagrangian fibrations along deformations.

Next, we recall the explicit description of small deformations of a hyperk\"ahler manifold $X$ via the local Torelli theorem.
\begin{rem}\label{rem:projective deformations}
Let $M$ be the universal deformation space of a hyperk\"ahler manifold $X$. By \cite[Prop.~22.11]{HuybrechtsHKM} we can identify $M$ near $[X]$ with the period domain
\[Q_X= \left\{[\sigma]\in\IP(H^2(X,\,\IC))\mid q(\sigma,\sigma)=0, \;q(\sigma,\bar{\sigma})>0\right\},\]
where $q$ is the Beauville--Bogomolov form introduced in Section~\ref{sect:q}. Now assume in addition that $X$ is projective.
By \cite[1.14]{huybrechts99} the subspace of $M$ consisting of those deformations $X_t$ of $X$ for which the class $[A] \in H^{1,1}\bigl(X,\,\IZ\bigr)$ of a given ample divisor $A$ remains of type $(1,1)$ (and hence $A$ continues to be an ample divisor) is given by $A^\perp=\{z \in Q_X \mid q(A,z)=0\}$. Consequently, there is a countable union of hypersurfaces in the period domain that parametrise those deformations of $X$ that are still projective.
\end{rem}
Since $X$ is simply--connected, $\mathrm{Pic}(X)$ injects into $H^2\bigl( X,\, \mathbb{C}\bigr)$. In the following we will hence not distinguish between divisors on $X$ and their classes in $H^2\bigl( X,\, \mathbb{C}\bigr)$.
\begin{lem}\label{lemma:char stab proj} Let $(X,L)$ be a projective hyperk\"ahler manifold together with a Lagrangian subtorus. We denote the inclusion of $L$ into $X$ by $j\colon L\into X$ and let
\[K=\ker\bigl(j^*\colon H^2\bigl(X,\, \IR\bigr)\to H^2\bigl(L,\,\IR\bigr)\bigr).\]
 Then the  following are equivalent.
\begin{enumerate}
\item Every small deformation of $(X,L)$ remains projective.
 \item There is an ample divisor $A$ on $X$ such that $A\in K^\perp_\IC$ or, equivalently, $K_\IC\subset A^\perp$.
\end{enumerate}
\end{lem}
\begin{proof}
We now prove $\refenum{i}\implies \refenum{ii}$ using the local description of the deformation space in terms of the period domain.
Deformations of $X$ that are induced by a deformation of the pair $(X,L)$ are given locally at $[X]$ by $K_\IC\cap Q_X$, see~\cite[Cor.~0.2]{voisin92}. Furthermore, this intersection is smooth, hence irreducible near $[X]$.  If all small deformations of $(X,L)$ remain projective, we can find an ample divisor $A$ on $X$ such that $K_\IC\cap Q_X\subset A^\perp\cap Q_X$, since an irreducible variety that is contained in a countable union of hypersurfaces is contained in one of them. We want to show $K_\IC\subset A^\perp$.

If $q\restr K$ was non--degenerate, $K_\IC\cap Q_X$ would be irreducible. In this case any small neighbourhood of $[X] \in K_\IC\cap Q_X \subseteq A^\perp$ would contain a basis of $K_\IC$ implying $K_\IC\subseteq A^\perp$. As $q$ and $K$ are defined over $\IR$, it is sufficient to prove non--degeneracy of $q \restr K$ over $\IR$. We have $H^{2,0}\oplus H^{0,2}\subseteq K_\IC$ by the Lagrange property, so if $q\restr{ K}$ was degenerate, it would be on $q\restr{ K \cap H^{1,1}(X)}$. Suppose there was $\delta \in K\cap H^{1,1}\bigl(X,\,\IR\bigr)$ with $q(\delta)=0$. 
But then $\delta$ would be contained in $A^\perp\cap H^{1,1}(X,\IR)$, where $q$ is negative definite; a contradiction.

If $\refenum{ii}$ holds, then we have $K_\IC\cap Q_X \subset A^\perp\cap Q_X $. Hence, $A$ remains an ample divisor on every small deformation of $(X,L)$, cf. Remark \ref{rem:projective deformations}. This implies $\refenum{i}$. 
%
\end{proof}


We will now pursue our idea to deform a pair $(X,L)$ to a non--projective pair, apply Theorem~\ref{mainnonproj} and then transport the fibration back along the deformation. 
\begin{lem}\label{lem:closed}
 Let
\[ \xymatrix{ \gothL \ar@{^(->}[rr]\ar[dr] && \gothX\ar[dl]^h\\& T}\]
be a family of pairs of a hyperk\"ahler manifold together with a Lagrangian subtorus (as in Definition~\ref{def: defo}) over a  one--dimensional complex disc $T$.
Assume that there is a dense subset $T'\subset T$ such that for all $t\in T'$ the fibre $\gothX_t$ admits a holomorphic Lagrangian fibration with fibre $\gothL_t$.
Then $\gothX_t$ admits an almost holomorphic Lagrangian fibration with strong fibre $\gothL_t$ for all $t\in T$.
\end{lem}
\begin{proof}
Recall that on a single pair $(X,L)$ there exists an almost holomorphic
Lagrangian fibration with strong  fibre $L$ if and only if the evaluation map in Diagram \eqref{diagram:Barlet} is bimeromorphic, cf.~Lemma~\ref{lem:epsilon birational}.
We will now exploit this in the relative setting. Let $\gothB (\gothX/T)$ be the relative Barlet space of $\gothX$ over $T$ and denote by $h_*\colon \gothB(\gothX/T)\to T$  the holomorphic map $Z\mapsto h(\mathrm{supp}(Z))$, cf.~\cite[Ch.~VIII, Thm~2.10]{SCVVII}.
\begin{custom}[Claim] The relative Barlet space $\gothB(\gothX/T)$ is smooth near $[\gothL_t]$ for all $t\in T$.\end{custom}
\begin{proof}[Proof of Claim]Since $T$ is a disc and hence does not contain compact positive--di\-men\-sio\-nal subvarieties, we obtain $\gothB(\gothX/T) = \gothB(\gothX)$ near $[\gothL_t]$. By Proposition~\ref{prop:DouadyBarlet} the space $\gothB(\gothX)$ is smooth at $[\gothL_t]$, if $\mathscr{D}(\gothX)_{\mathrm{red}}$ is smooth. Again using that $T$ is a disc, we have $ \mathscr{D}(\gothX)_{\mathrm{ red}}= \mathscr{D}(\gothX/T)_{\mathrm{ red}}$ near $[\gothL_t]$.
However, as $T$ is smooth, the relative Douady space $\mathscr{D}(\gothX/T)$ (and therefore its reduction) is smooth at $[\gothL_t]$ by \cite[Prop.~2.1]{matsushita09}, see also \cite[Prop~2.4]{voisin92} and \cite[Cor. VI.6.3]{LehnDiss}. This shows the claim.
\end{proof}
Note that the family $\gothL \to T$ induces a section $\sigma \colon T \hookrightarrow \gothB(\gothX /T)$ of $h_*$. As $\gothB(\gothX/T)$ is smooth along $\sigma(T) = \{[\gothL_t] \mid t\in T \}$, there exists a unique irreducible component $\gothB(\gothX/T)_0$ of $\gothB(\gothX/T)$ containing $ \sigma(T)$. Since for any $t \in T$, the Barlet space $\gothB(\gothX_t)$ of the fibre is smooth at $[\gothL_t]$ by Lemma~\ref{deflem}, the (reduction of the) fibre of $h_*$ over $t$ contains the unique irreducible component $\gothB_t$ of $\gothB(\gothX_t)$ containing $[\gothL_t]$. We set $g:= h_*\restr{\gothB(\gothX/T)_0}$. Possibly covering $T$ by several smaller discs and treating these separately we may assume $g\colon \gothB(\gothX/T)_0\to T$ to be proper by Proposition~\ref{prop: folklore}.

The evaluation map from the restriction of the graph of the universal family $\gothU_T$ to $\gothB(\gothX/T)_0$ gives a commutative diagram
\[\xymatrix{ &\gothU_T\ar[dr]^{\bar\epsilon}\ar[dl]_{\bar\pi}\ar[dd]^{g'}\\
\gothB(\gothX/T)_0 \ar[dr]_g&& \gothX\ar[dl]^h\\
&T.}
\]
As $g$ admits a section, it is surjective.
Note that for all $t\in T$ the fibre $(\gothB(\gothX/T)_0)_t$ set--theoretically coincides with the union of some components of the cycle space $\gothB(\gothX_t)$; one of these components is equal to $\gothB_t$. Additionally, the reduced fibre $g'^{-1}(t)_{\rm red}$ coincides with the graph of the universal family over these components, and the restriction of $\bar \epsilon$ to $g'^{-1}(t)_{\rm red}$ is the (absolute) evaluation map.

Since  $g$ is a surjective and proper map from an irreducible space onto a (smooth) 1--dimensional disc, it is flat by \cite[Ch.~II, Thm.~2.9]{SCVVII}.
By \cite[\S 3, Prop.~2]{fujiki82} there  is a finite covering $\beta\colon\widetilde{T}\to T$, a reduced and irreducible closed subspace $\widetilde{\gothB}$ of $\gothB(\gothX/T)_0\times_T\widetilde{T}$ such that the induced map $\alpha\colon\widetilde{\gothB}\to\gothB(\gothX/T)_0$ is bimeromorphic, and a non--empty Zariski--open subset $V$ of $\widetilde{T}$ such that $\widetilde{g}\colon\widetilde{\gothB} \to \widetilde{T}$ has irreducible fibres over $V$. The proof of \cite[\S 3, Prop.~2]{fujiki82} shows that we can choose $\alpha$ to be an isomorphism over the normal locus of $\gothB(\gothX/T)_0$. As the section $\sigma\colon T\to\gothB(\gothX/T)_0$ maps to the smooth locus of $\gothB(\gothX/T)_0$ by the Claim proven above, it lifts to  a section $\widetilde{\sigma}\colon T\to\widetilde{\gothB}$ of $\beta \circ \widetilde g $. Composing with $\widetilde g$ we obtain a section of $\beta$. As $\beta$ is finite and $\widetilde{T}$ is irreducible, it follows that $\beta$ is an isomorphism. Summing up, we find a non--empty Zariski--open subset $V$ of $T$ such that $g$ has irreducible fibres over $V$. Consequently, for $t \in V$ the reduced fibre $g^{-1}(t)_{\rm red} $ coincides with the irreducible component $\gothB_t$ of $\gothB(\gothX_t)$ containing $[\gothL_t]$ and the reduced fibre ${g'}^{-1}(t)_{\rm red}$ is the graph of the universal family over this component.

As a consequence, the evaluation map $\epsilon_t\colon {g'}^{-1}(t)_{\rm red} \to \gothX_t$ is generically finite for all $t \in V$. From this we infer that $\bar \epsilon\restr{\gothU_V}$ is likewise generically finite. For $t\in T'\cap V$ we know that $\bar\epsilon_t \colon \gothU_t \to \gothX_t$ is an isomorphism by Lemma~\ref{Chowiso}. Hence, testing its degree at a general point in a fibre over some general $t\in T' \cap V$ we see that $\bar \epsilon\restr{\gothU_V}$ is in fact bimeromorphic. Thus, also $\bar\epsilon$ is bimeromorphic. Since the property of $\bar\epsilon_t$ being an isomorphism is open in the base $T$, after shrinking $T$ if necessary we may assume that $\gothX_t$ admits a Lagrangian fibration with fibre $\gothL_t$ for $t\neq 0$. Thus, $\bar\epsilon\colon \gothU_T\to \gothX$ is birational and an isomorphism on $\inverse{g'} (T\setminus\{0\})$.

If $\bar\epsilon$ is an isomorphism, the claim follows from Lemma~\ref{Chowiso}; so assume that this is not the case. We decompose the central fibre  into irreducible components $\inverse{g'}(0)_\mathrm{red}=\gothU_0 \cup \bigcup_i \gothV_i$, where $\gothU_0$ is the graph over $\gothB_0$,  the  component of the Barlet space of $\gothX_0$ containing $[\gothL_0]$. Restricting $\bar\epsilon$ to $\gothU_0$, we recover the absolute evaluation map for $\gothX_0$ as considered in Diagram \eqref{diagram:Barlet}, which is generically finite by Lemma~\ref{deflem}.

Since the global map $\bar\epsilon$ is proper birational, and $\gothX$ is smooth, we conclude that $\epsilon_0$ is a proper generically finite map of degree 1 onto the normal space $\gothX_0$, thus bimeromorphic.
Using Lemma~\ref{lem:epsilon birational} we conclude that $\gothX_0$ admits an almost holomorphic Lagrangian fibration with strong fibre $\gothL_0$. This completes the proof of Lemma~\ref{lem:closed}.
\end{proof}

\begin{proof}[Proof of Theorem~\ref{notstabproj}]
To show $\refenum{ii}\implies \refenum{i}$, let us assume that there is a small deformation $(\gothX_t, \gothL_t)_{t\in T}$ of the pair $(X, L) = (\gothX_0, \gothL_0)$ over a disc $T$ such that not all $\gothX_t$ are projective. We have seen in  Remark~\ref{rem:projective deformations} that there exists a dense subset $T'$ of $T$ such that the fibre $\gothX_t$ is non--projective for all $t\in T'$.

For all $t\in T'$ the non--projective hyperk\"ahler manifold $\gothX_t$ admits a holomorphic Lagrangian fibration with fibre $\gothL_t$ by Theorem~\ref{mainnonproj}. Thus, the family satisfies the assumptions of Lemma~\ref{lem:closed}, and we conclude that $X$ admits an almost holomorphic Lagrangian fibration with strong fibre $L$.

For $\refenum{i}\implies \refenum{iii}$, assume that $f\colon X\dasharrow B$ is an almost holomorphic Lagrangian fibration with strong fibre $L$ on the projective hyperk\"ahler manifold $X$. After a suitable modification, we may assume that $B$ is also projective. Then, the pullback $D = f^*(A)$  of a general ample divisor $A$ on $B$, defined as in Remark~\ref{rem: pullback}, is trivial on $L$.


Finally, assuming \refenum{iii} suppose that every small deformation of $(X, L)$ remains projective. Choose an ample divisor $A$ as in Lemma~\ref{lemma:char stab proj}~\refenum{ii}. Since $D$ restricts to zero on $L$, we have $D\in K \subset A^\perp$. However, if $A'$ is any ample divisor and $D'$ is an arbitrary divisor on $X$, then the Beauville--Bogomolov form $q(A',D')$ computes the intersection number $D'.(A')^{n-1}$ up to a non--trivial factor, cf.~\cite[Exerc.~23.2]{HuybrechtsHKM}. Consequently, we obtain $q(A, D) \neq 0$, which contradicts $D \in A^\perp$.
\end{proof}

\section{Holomorphic models for almost holomorphic Lagrangian fibrations}\label{sect:holmodels}
The aim of this section is to prove that an almost holomorphic Lagrangian fibration can be modified to give a holomorphic Lagrangian fibration on a possibly different hyperk\"ahler manifold. This proves a special case of the Hyperk\"ahler SYZ--Conjecture.

\subsection{Almost holomorphic fibrations and the SYZ--Conjecture}\label{sect:SYZ}
We have seen in Corollary~\ref{cor: almost holomorphic non projective} that the existence of an almost--holomorphic Lagrangian fibration on a non--projective hyperk\"ahler manifold implies the existence of
a holomorphic Lagrangian fibration. If $X$ is projective, the question of passing from ``almost--holomorphic'' to ``holomorphic''
is connected to a circle of well--known conjectures\footnote{See for example \cite{verbitsky10} and the references given therein.}:

\begin{custom}[Hyperk\"ahler SYZ--Conjecture] Let $X$ be a hyperk\"ahler manifold.
\begin{enumerate}
 \item If  $\kl$ is a non--trivial  line bundle   on $X$ such that $q(\kl)=0$ and $\kl$ is nef, then $\kl$ is semi-ample and a  suitable power of it gives rise to a holomorphic Lagrangian fibration on $X$.
\item If $\kl$ is a non--trivial  line bundle on $X$ such that $q(\kl)=0$, then there is a hyperk\"ahler manifold $X'$ and a bimeromorphic map $\phi\colon X'\dashrightarrow X$ such that $\phi^*(\kl)$ is nef. The map $\phi$ is expected to be a composition of Mukai flops.
\end{enumerate}
\end{custom}
This conjecture relates to the simpler version mentioned in the introduction in the following way: using the description of the universal deformation space via the period map one can deform any hyperk\"ahler manifold of second Betti number at least 5 to a hyperk\"ahler manifold that admits a non--trivial nef line bundle  with $q(\kl)=0$ \cite[Prop.~4.3]{sawon03}. Currently, there are no examples of hyperk\"ahler manifolds known with smaller second Betti number.
The claims of the Hyperk\"ahler SYZ---Conjecture have been established under a variety of additional assumptions, see for example \cite{matsushita08, ame-cam08, verbitsky10, cop}.

The following result shows the relation to almost holomorphic La\-gran\-gian fibrations.
\begin{prop}\label{prop:qnef}
Assume that $f\colon X\dashrightarrow B$ is an almost holomorphic Lagrangian fibration with $X$ and $B$ projective. Let $A$ be a general very ample divisor on $B$ and $D:=f^*(A)$, defined as in Remark~\ref{rem: pullback}. Then, the following holds:
 \begin{enumerate}
 \item  $q(D)=0$,
 \item if  $D$ is nef, then $\ko_X(D)$ is semi-ample and induces a holomorphic Lagrangian fibration $f'\colon X \to B'$.
 \end{enumerate}
\end{prop}

\begin{proof}[Proof of Proposition~\ref{prop:qnef}]
The fact that $q(D)=0$ is proved in \cite[Proposition 3.1]{ame-cam08}. The idea is that $q(D)\geq 0$, since $D$ has no fixed component. Furthermore, if $q(D)>0$ then it would be big (see \cite[proof of Prop.~26.13]{HuybrechtsHKM} and use for example \cite[Prop.~6.6 (f)]{DemaillyICTP}). This is impossible, since $D$ induces a fibration onto a lower--dimensional space.

If in addition $D$ is nef, it is straightforward to see that $f$ coincides with the nef reduction of the line bundle $\ko_X(D)$, defined in \cite[Theorem 2.1 and Definition 2.7]{bcekprsw}. Thus, the nef dimension of $f^*(A)$ is $\dim X/2< \dim X$, and we conclude by \cite[Theorem 1.5]{matsushita08}.
\end{proof}

\subsection{Good minimal models for almost holomorphic Lagrangian fibrations}
Consider again an almost holomorphic Lagrangian fibration $f\colon X \dasharrow B$ on a hyperk\"ahler manifold $X$. Assume that the pull--back $D=f^*(A)$ of a general very ample divisor
 on $B$ is not nef, so Proposition \ref{prop:qnef} does not apply. Since $K_X$ is trivial, $D$ is an adjoint divisor. Consequently, it is a natural idea to run a minimal model program (see \cite{Kollar-Mori} for an introduction and the standard terminology) in order to make $D$ nef. Proceeding in this way we will prove the following result.

\begin{theo}\label{thm: holomorphic model}
 Let $X$ be a projective hyperk\"ahler manifold with an almost holomorphic Lagrangian fibration $f\colon X \dasharrow B$ over a normal projective variety $B$. Let $A$ be a general very ample divisor on $B$ and $D = f^*(A)$. Then, there exists a hyperk\"ahler manifold $X'$ and a birational map $\phi\colon X \dasharrow X'$ with the following properties
 \begin{enumerate}
\item $\phi$ is the composition of a finite sequence of $K_X$--flops; in particular, it is an isomorphism in codimension one,
 \item no center of $\phi$ intersects any general fibre of $f$,
 \item $\phi_*(D)$ is nef, and the linear system $|\phi_*(D)|$ defines a holomorphic Lagrangian fibration $f'\colon X' \to B'$, where $B'$ is a normal projective variety birational to $B$.
 \end{enumerate}
 \end{theo}
Although it is our aim to construct a \emph{smooth} model $X'$ for $X$, intermediate steps in the construction might introduce singularities. We therefore recall the following terminology from \cite[Def 1.1]{beauville00}.
\begin{defin}[\textbf{Beauville}]
 A normal projective variety (or compact K\"ahler space) $Z$ is called a \emph{symplectic variety}, if there is a symplectic form $\sigma$ on $Z_\sm$ that extends to a holomorphic 2--form on some resolution of $Z$.
\end{defin}
As noted in \cite[(1.2)]{beauville00}, the form $\sigma$ extends to one resolution if and only if it extends to every resolution.
The proof of Theorem \ref{thm: holomorphic model} is based on recent advances in the minimal model program. We will start by proving that the final outcome is smooth by a slight generalisation of a result of Namikawa.
\begin{prop}\label{prop: general namikawa}
 Let $X$ be a hyperk\"ahler manifold and $X'$ a symplectic variety birational to $X$. If $X'$ has $\IQ$--factorial terminal singularities, then $X'$ is a smooth hyperk\"ahler manifold. Furthermore, $X$ and $X'$ are birational via a finite number of $K_X$--flops.
\end{prop}
\begin{proof}
If $\sigma'$ is a symplectic form on $X'_\sm$ then $\sigma'^{\dim X/2}$ is a nowhere vanishing generator of the canonical bundle $\omega_{X'_\sm}$, so $K_{X'}$ is trivial and in particular nef.
Therefore, both  $X$ and $X'$ are minimal models of a common desingularisation, and thus they are connected by a finite number of flops by \cite{kawamata08}. The claim then follows from \cite[p.~98, Example]{namikawa06}.
\end{proof}

Theorem \ref{thm: holomorphic model} now follows immediately from  Theorem \ref{thm: matsushita} and the following Lemma whose proof will occupy the rest of this section.

\begin{lem}[Holomorphic Model Lemma]\label{lem:mmp}
Let $X$ be a projective hyperk\"ahler manifold, and let $f\colon X \dasharrow B$ be a dominant almost holomorphic map with connected fibres to a complex space $B$ with $\dim B < \dim X$. Then there exists a holomorphic model for $f$ on a hyperk\"ahler manifold $X'$ birational to $X$; that is, there exists a diagram
\[ \xymatrix{X\ar@{-->}[r]^{\phi}\ar@{-->}[d]_f& X'\ar[d]^{f'}\\B\ar@{-->}[r] & B'}\]
such that $B'$ is a normal projective variety bimeromorphic to $B$, $f'$ is holomorphic, and $X'$ is hyperk\"ahler manifold. The map $\phi$ is an isomorphism near the general fibre of $f$.

Moreover, if $D:= f^*(A)$ is the pull--back of a general very ample divisor on $B$, then $\phi_*(D)$ is nef, and $f'$ can be chosen to be
the map associated to the linear system of $\phi_*(D)$.
\end{lem}
\begin{rem}
A slightly more general result has been known to be true in dimension four due to work of Amerik and Campana \cite[Thm.~3.6]{ame-cam08}. Compare also with \cite{wierzba}.
\end{rem}
\subsection{Proof of the Holomorphic Model Lemma}
Let $X$ be a projective hyperk\"ahler manifold, and let $f\colon X \dasharrow B$ be a dominant almost holomorphic map with connected fibres to a complex space $B$ with $\dim B < \dim X$. We may assume that $B$ is projective and denote by $D = f^*(A)$ the pullback of a  general very ample divisor $A$ on $B$.

\subsubsection{Setting the stage}
 Choose a small, positive, rational number $\delta$ such that with $\Lambda = \delta D$ the pair $(X, \Lambda)$ is klt. This is possible since $X$ is smooth and $D$ is effective. Note that $K_X+\Lambda=\Lambda$, since $K_X$ is trivial.

 By \cite[Cor.~1.1.2]{BCHM} or \cite[Thm.~1.1]{cascini-lazic10} the adjoint ring
\[R:=R(X, K_X+\Lambda)=\bigoplus_{m\in \IN} H^0(X, \ko(\lfloor m\Lambda \rfloor))\]
is finitely generated, and we can therefore choose a positive number $d$ such that $d\Lambda$ is integral and such that the Veronese subring $R'=R^{(d)}=\bigoplus_{m\in \IN} R_{md}$ is generated in degree 1.

Since $md\Lambda$ is the pullback of a general very ample divisor on $B$, the general member of the linear system $|md\Lambda|$ is irreducible. It follows that the base locus of $|md\Lambda|$ is a subset of at least codimension 2, not intersecting the general fibre of $f$. As $R'$ is generated in degree 1, $\Bs(d\Lambda)=\Bs(md\Lambda)$ for all positive integers $m$.

We now choose a log--resolution $p\colon \widetilde X\to X$ of the linear series $|d\Lambda|$ (see e.g. \cite[p.144]{LazarsfeldII}) such that the exceptional divisor of $p$ maps to the base locus of $|d\Lambda|$. By the above there are divisors $M$ and $G$ on $\widetilde X$ such that for all $m\in \IN$
\begin{equation}\label{eq:linsys}
 p^*|md\Lambda|=|p^*(md\Lambda)|= |mM| +mG,
\end{equation}
$mM$ is basepoint free and $mG$ is the fixed part of $|p^*(md\Lambda)|$. By our choice of the resolution, we have $\mathrm{supp}(G)=\mathrm{Exc}(p)$.

Let $\widetilde f\colon \widetilde X\to \widetilde B$ be the Stein factorisation of the morphism associated to the linear system $|M|$. The composition $\widetilde f\circ \inverse p$ is still almost holomorphic so we may replace $B$ by $\widetilde B$ resulting in a diagram
\[\xymatrix{ & \widetilde X \ar[dl]_p \ar[dr]^{\widetilde f}\\ X \ar@{-->}[rr]^f && B.}\]

Let $E:=\sum_i E_i$ be the reduced exceptional divisor of $p$, and let $a_i:= \mathrm{discr}(E_i, X, \Lambda )$ be the discrepancies with respect to $\Lambda$. By definition and the klt--assumption the effective divisors
\begin{equation*}
\widetilde \Lambda := \inverse p_* \Lambda + \sum_{-1< a_i<0} -a_i E_i \text{\hspace{0.7cm} and \hspace{0.7cm}} F:=\sum_{ a_i>0} a_i E_i
\end{equation*}
do not have common components and satisfy
$ K_{\widetilde X} + \widetilde \Lambda = p^* (K_X+ \Lambda) +F$.
We have linear equivalences
\begin{equation}\label{eq:zerlegung}
d(K_{\widetilde X} + \widetilde\Lambda )\sim d p^* \Lambda+ dF
\sim  M+G+ dF,
\end{equation}
where $G+ dF$ is effective, and $\mathrm{supp}(G + dF) = E$ is
the stable base locus of  $K_{\widetilde X} + \widetilde\Lambda $ (cf.~\eqref{eq:linsys}).

\subsubsection{Running an MMP}
Let  $U\subset B$ be the largest open subset such that the restriction $f\restr{\inverse{f(U)}}\colon \inverse{f}(U) \to U$ is proper holomorphic and set $\widetilde U := \inverse{\widetilde f}(U)$.
Then, $p$ induces an isomorphism from $\widetilde U$ to $\inverse{f}(U)$, which implies in particular that $K_{\widetilde U}$ is trivial. Since furthermore the restriction of $\widetilde \Lambda$ to $\widetilde U$ coincides with the pull--back of $A \cap U$ via $\widetilde f|_{\widetilde U}$, the restriction of $K_{\widetilde U} + \widetilde \Lambda|_{\widetilde U}$ to the fibre $\widetilde f^{-1}(b)$ is nef for all $b\in U$. Consequently, $(\widetilde U, \widetilde \Lambda|_{\widetilde U})$ is a good minimal model over $U$ in the sense of \cite[Sect.~2]{HaconXu}.

Since moreover the pair $(\widetilde X, \widetilde \Lambda )$ is klt by definition, it  has a good minimal model over $B$ by a recent result of Hacon and Xu \cite[Thm.~1.1]{HaconXu}. Let $H$ be an ample divisor on $\widetilde X$. Then it follows from \cite[Prop.~2.5]{lai11} that any minimal model program over $B$ with scaling by $H$ (see for example \cite[Rem.~3.10.10]{BCHM}) terminates in a good minimal model $(\hat X, \hat \Lambda)$ for $(\widetilde X, \widetilde \Lambda)$ over $B$. We summarise the situation in the following diagram:
\[\begin{xymatrix}{&\ar[dl]_p \widetilde  X\ar[dd]^{\widetilde f} \ar@{-->}[rd]^{\psi}\\
X\ar@{-->}[dr]_f & &  \hat X\ar[ld]^{\hat f}\\
  &  B   & .
}
\end{xymatrix}
\]
\subsubsection{Analysing the outcome of the MMP}
Recall that the relative stable base locus of an effective $\mathbb{Q}$-divisor $D$ on $\hat X$ is defined as
\[
 \mathbb{B}(D/B) =  \bigcap_{D'\geq 0, \; D\sim_{\IQ,B} D'} \supp(D') 
\]
where $D \sim_{B,\IQ} D' :\iff D \sim_\IQ D' + \hat f^*(C)$ for some $\IQ$-Cartier divisor $C$ on $B$. The relative diminished stable base locus is defined as 
\[
 \mathbb{B}_-(D/B)=  \bigcup_{\varepsilon \in \mathbb{Q}^{>0} } \mathbb{B}(D+\varepsilon A/B),
\]
where $A$ is any ample $\mathbb{Q}$-divisor; see for example \cite[Sect.~2.E]{HaconKovacs}. We have $\mathbb{B}_-(D/B) \subset \mathbb{B}(D/B)$. 

\begin{custom}[Claim]
The pair $(\hat X, \hat \Lambda)$ is a good minimal model for $(\widetilde X,\widetilde\Lambda)$ and $\hat f$ is the map associated with the linear system of $\hat M := \psi_* M = d(K_{\hat X} + \hat \Lambda) = d \delta \hat f^* A$. In particular, the $p$-exceptional set $E$ is contracted by $\psi$.
\end{custom}
\begin{proof}
The non-trivial part of the proof is identical to the one of \cite[Thm.~4.4]{lai11}. We will give a sketch of the argument for the reader's convenience.

As $(\hat X, \hat \Lambda)$ is a good minimal model for $(\widetilde X,\widetilde\Lambda)$ over $B$, the relative stable base locus $\IB(K_{\hat X}+\hat \Lambda/B)$ is empty. In particular, the restricted base locus $\IB_-(K_{\hat X}+\hat \Lambda/B)$ is empty. Moreover, the pushforward $\hat G + d \hat F= \psi_*(G+dF)$ remains vertical, as $\psi$ is an isomorphism on $\widetilde U$, and we have $\IB(K_{\hat X}+\hat \Lambda)=\supp(\hat G + d \hat F)$. Thus, the argument of \cite[Thm.~4.4]{lai11} shows that every effective divisor $T$ on $\hat X$ with $\supp (T) \subset \supp(\hat G+d\hat F)$ is degenerate in the sense of \cite[Def. 2.9]{lai11}. Following Lai, we conclude by \cite[Lem.~2.10]{lai11} that $T$ is contained in $\IB_-(K_{\hat X}+\hat \Lambda/B)$, which shows that there exists no such component.
\end{proof}

Set $\hat \phi:= \psi \circ p^{-1}\colon X\dasharrow \hat X$. Note that $\hat\phi$ is an isomorphism over $U$.
\begin{custom}[Claim]
$\hat X$ is a $\mathbb{Q}$-factorial symplectic variety; in particular, $K_{\hat X}$ is trivial.
\end{custom}
\begin{proof}
Since $\hat X$ is the outcome of a minimal model program that started on the $\mathbb{Q}$-factorial variety $\widetilde X$, it is likewise $\mathbb{Q}$-factorial, e.g.~see~\cite[Prop.~3.36(1) and Prop.~3.37(1)]{Kollar-Mori}.

Since $E$, the exceptional locus of $p$, is contracted by $\psi$, the varieties $X$ and $\hat X$ are isomorphic in codimension one. As a consequence, the canonical sheaf $\omega_{\hat X}$ of $\hat X$ is trivial. Furthermore, the symplectic form $\sigma$ on $X$ induces a symplectic form $\hat \sigma$ on a smooth open subset $V \subset \hat X$ with complement of codimension two in $\hat X$. Since the sheaf of holomorphic 2--forms on $\hat X_\sm$ is locally free, $\hat \sigma$ extends to a holomorphic $2$--form $\hat\sigma'$ on $\hat X_\sm$. As $K_{\hat X}$ is trivial, the extended form $\hat \sigma'$ remains everywhere non--degenerate.
Taking a smooth resolution of the indeterminacies of $\psi$, we obtain a common resolution of $X$ and $\hat X$ to which $\hat \sigma'$ extends as a holomorphic 2--form. Consequently, $\hat X$ is a symplectic variety as claimed.
\end{proof}

\begin{custom}[Claim]
$\hat X$ has terminal singularities; hence, as a consequence of Proposition~\ref{prop: general namikawa} it is a smooth hyperk\"ahler manifold.
\end{custom}
\begin{proof}
 As noted already in the proof of the previous claim, the varieties $\hat X$ and $X$ are isomorphic in codimension one. Hence, if $Z$ is a common resolution of singularities 
$$ \begin{xymatrix}{
   & Z \ar[rd]^{\hat \pi} \ar[ld]_{\pi}&  \\
X &  & \hat X 
}
   \end{xymatrix}
$$
of $\hat X$ and $X$, every $\hat \pi$-exceptional divisor is $\pi$-exceptional and vice versa. Therefore, writing out the discrepancy formula for both resolutions we obtain
$$ K_Z \;\equiv\; \underset{= 0}{\underbrace{\pi^*(K_{X})}} + \sum_{E \text{ $\pi$-exc.}} \underset{\geq 1}{\underbrace{a(E, X)}} E \;\;= \;\; \underset{= 0}{\underbrace{\hat\pi^*(K_{\hat X})}}+ \sum_{E \text{ $\hat \pi$-exc.}}a(E, X) E. $$Consequently, $\hat X$ is terminal, as claimed.
\end{proof}

\begin{custom}[Claim]
The divisor $\hat \phi_*(D)$ is linearly equivalent to ${\hat f}^*(A)$; in particular, it is nef and basepoint--free.
\end{custom}
\begin{proof}
We have seen above that $\hat\phi\colon X\dasharrow \hat X$ is an isomorphism in codimension one. It follows that
\[ \hat \phi_*(D) = \hat\phi_*(f^*(A))= \hat f ^*(A).\]
In particular,  the divisor $\hat \phi_*(D)$  is nef, and the map associated to $|\hat \phi_*(D)|$ factors as \[\hat X \overset{\hat f}{\longrightarrow} B \into \IP^N. \qedhere\]
\end{proof}
This concludes the proof of Lemma~\ref{lem:mmp}.

\subsection{Further applications of the Holomorphic Model Lemma}
Lemma \ref{lem:mmp} allows to generalise some results on fibrations on hyperk\"ahler manifolds to the almost holomorphic case. As an example we give a generalisation of Matsushita's results summarised in Theorem \ref{thm: matsushita}.

\begin{theo}\label{thm: almost holomorphic is Lagrangian}
 Let $X$ be a projective hyperk\"ahler manifold and $f\colon X \dasharrow B$ an almost holomorphic map with connected fibres onto a normal projective variety $B$ such that  $0<\dim B< \dim X$.  Then $\dim B=\frac 1 2 \dim X$, and $f$ is an almost holomorphic Lagrangian fibration.
\end{theo}
\begin{proof}
By Lemma \ref{lem:mmp} there is a holomorphic model $f'\colon X'\to B' $ on a possibly different hyperk\"ahler manifold birational to $X$, and isomorphic to $X$ near the general fibre of $f'$. By Matsushita's theorem (see Theorem \ref{thm: matsushita})  $\dim B = \dim B' = \frac 1 2 \dim X$, and $f'$ is a Lagrangian fibration. Thus, $f$ is  an almost holomorphic Lagrangian fibration.
\end{proof}

\begin{rem}
Note that Theorem \ref{thm: holomorphic model} proves the second part of the Hyperk\"ahler SYZ--Conjecture (see Section \ref{sect:SYZ}) under the additional assumption that  some multiple of $\kl$ defines an almost holomorphic fibration. Indeed, in this case Theorem \ref{thm: almost holomorphic is Lagrangian} implies that we get an almost holomorphic Lagrangian fibration on $X$ and Theorem \ref{thm: holomorphic model} implies that we can find a holomorphic Lagrangian fibration on a birational hyperk\"ahler manifold $X'$. This implies the statement of the conjecture, because birational hyperk\"ahler manifolds are deformation equivalent \cite[Thm.~4.6]{huybrechts99}.
\end{rem}

\begin{rem}\label{rem: further devel} 
 Note that the holomorphic model lemma, Lemma~\ref{lem:mmp}, yields a good minimal model of the pair $(X, D)$. Consequently, any minimal model program with scaling for $(X,D)$  will terminate in a good minimal model. Matsushita has announced an argument \cite{MatsushitaPreprint}, which deduces from termination of a minimal model program that any almost holomorphic Lagrangian fibration on a projective hyperk\"ahler manifold becomes holomorphic after a suitable modification of the base (as in Lemma~\ref{lem:mmp}). Together with the result of Hwang and Weiss \cite{HwangWeiss} mentioned in the introduction 
and Theorem~\ref{notstabproj} this would provide a complete positive answer to Beauville's question also in the projective case. 
\end{rem}

\newcommand{\etalchar}[1]{$^{#1}$}


\end{document}